\documentclass[11pt,a4paper,twoside,draft]{article}
\usepackage{amssymb,amsthm}
\usepackage[namelimits,sumlimits,fleqn]{amsmath}
\usepackage{setspace}
\usepackage[a4paper,top=33mm, bottom=27mm, left=20mm, right=20mm]{geometry}
\usepackage[small, margin=20pt]{caption}
\usepackage{float}\restylefloat{figure}
\usepackage{url}
\usepackage[final]{graphicx}
\usepackage{subfig,multicol}

\usepackage{fancyhdr}
\allowdisplaybreaks[4]

\title{A probabilistic approach to consecutive pattern avoiding in permutations
 \thanks{The author wants to thank
the FPU grant from \emph{Ministerio de Educaci\'on de Espa\~{n}a}.}}
\author{Guillem Perarnau\\
\\
  \emph{Departament de Matem\`atica Aplicada $IV$.}\\ 
  \emph{Universitat Polit\`ecnica de Catalunya, BarcelonaTech.}\\
  \url{guillem.perarnau@ma4.upc.edu}\\}
\date{\today}

\theoremstyle{plain}
\newtheorem{theorem}{Theorem}%[section]
\newtheorem{lemma}[theorem]{Lemma}

\newtheorem{corollary}[theorem]{Corollary}

\theoremstyle{definition}
\newtheorem{conjecture}[theorem]{Conjecture}

\newtheorem*{acknowledgement}{Acknowledgement}

\theoremstyle{definition}

\newcommand{\rst}[1]{\ensuremath{{\mathbin\upharpoonright}%
\raise-.5ex\hbox{$#1$}}} % creates restriction symbol

\renewcommand{\S}{\mathcal{S}}
\newcommand{\M}{\mathcal{M}}
\newcommand{\N}{\mathcal{N}}
\newcommand{\A}{\mathcal{A}}

\newcommand{\T}{\mathcal{T}}
\newcommand{\supp}{supp}
\newcommand{\eps}{\varepsilon}
\newcommand{\st}{\mathrm{st}}

\begin{document}

\pagenumbering{arabic}

\setcounter{section}{0}

\bibliographystyle{plain}

\maketitle

\onehalfspace

\thispagestyle{empty}
 
\begin{abstract}
We present a new approach to the problem of enumerating permutations of length $n$ that avoid a
fixed consecutive pattern of length $m$. We use this idea to give explicit upper and lower
bounds on the number of permutations avoiding a pattern of length $m$. As a corollary, we obtain a
simple proof of
the CMP conjecture~\cite{en2003}, regarding the most avoided pattern, recently shown by
Elizalde~\cite{e2012}. Finally, we also show that most of the patterns behave similar to the least
avoided
one.
\end{abstract}


\section{Introduction}

Let $\S_n$ be the symmetric group of permutations of length $n$. Consider a permutation $\pi=
(\pi_1,\dots,\pi_n)\in \S_n$ and a \emph{pattern} $\sigma=
(\sigma_1,\dots,\sigma_m)\in \S_m$. Henceforth $m$ will be a fixed integer while $n$ will be
considered to tend to infinity. For any sequence of different
positive integers $X=(x_1,\dots ,x_k)$, we define the \emph{standardization} of $X$,
$\st(x_1,\dots, x_k)$, as the permutation in $\S_k$ obtained by relabeling each element
of $X$ on the set $\{1,\dots ,k\}$ such that the order among all the elements of $X$ is
preserved.

A permutation $\pi\in\S_n$ \emph{contains $\sigma$ as a consecutive pattern} if there exists $0\leq
i\leq n-m$ such that $\st(\pi_{i+1},\dots,\pi_{i+m}) = \sigma$, this is, there are $m$ consecutive
elements
in $\pi$ that have the relative order prescribed by $\sigma$. For instance, if $\sigma=(12\dots m)$,
$\pi$ contains $\sigma$ as a consecutive pattern if and only if it contains $m$ consecutive
increasing elements (a run of length $m$). A pattern $\pi\in\S_n$ is called \emph{$\sigma$-avoiding}
if it
does not contain $\sigma$ as a consecutive pattern. Denote by $\alpha_n(\sigma)$, the number of
permutations in $\S_n$ that are $\sigma$-avoiding.

The problem of determining $\alpha_n(\sigma)$ is inspired by the problem of finding the number of
permutations of length $n$ that avoid a pattern $\sigma$ non necessarily in consecutive positions. A
permutation  $\pi\in\S_n$  \emph{contains} $\sigma$ if there exist $1\leq i_1<\dots<i_m\leq n$ such
that
$\st(\pi_{i_1},\dots,\pi_{i_m})=\sigma$. Clearly, if $\pi$ avoids $\sigma$, then $\pi$ also avoids
$\sigma$ as a consecutive pattern. Knuth~\cite{k1973} introduced the latter problem and exactly
determined the number of permutations avoiding some pattern of length $3$. There are many
interesting results in the area~(see e.g.~\cite{b1997,af2000}) as well as the famous Stanley-Wilf
conjecture which was solved by Marcus and Tardos~\cite{mt2004}.

To provide an exact formula for $\alpha_n(\sigma)$ is a tough problem when $n$ becomes large.
However, asymptotic formulas can be derived as shown by Elizalde and Noy in~\cite{en2003}.
They
provide an estimation of $\alpha_n(\sigma)$ for any pattern $\sigma$ of length $3$ and
also for some patterns of length $4$. Nowadays, an asymptotic formula of $\alpha_n(\sigma)$ is
not even known for all the patterns of length $4$.

Elizalde in~\cite{e2006} showed that for any $\sigma\in\S_m$, the following
limit exists,
$$
\rho_{\sigma}= \lim_{n\to \infty} \left(\frac{\alpha_n(\sigma)}{n!}\right)^{1/n},
$$
and that $0.7839<\rho_{\sigma}<1.$ In particular, it is known that
$\alpha_n(\sigma)\sim c\,\rho_{\sigma}^n\,n!$, for some constant $c$ that depends on $\sigma$.

Whereas it is very hard to exactly compute $\rho_{\sigma}$ for any $\sigma\in\S_m$, it is possible
to
provide upper
and lower bounds in terms of $m$. Besides, it is interesting to see which patterns are
extremal in that sense.
A pattern of length $m$ is called \emph{monotone} if it is either $(12\dots m)$ or $(m\dots 21)$.
It is clear that $\alpha_n(12\dots m)=\alpha_n(m\dots 21)$, since $\pi\in \S_n$ is $(12\dots
m)$-avoiding if and only if its reversing $\tilde{\pi}=(\pi_n, \dots,\pi_1)$ is $(m\dots
21)$-avoiding.
In~\cite{en2003} it was conjectured that monotone patterns are the
most avoided ones among all patterns of length $m$, when $n$ is large enough. This is known as the \emph{Consecutive
Monotone Pattern (CMP) Conjecture}.
\begin{conjecture}[CMP conjecture~\cite{en2003}]\label{conj:most avoided pattern}
 For any $\sigma\in \S_m$,
$$
\rho_{\sigma}\leq \rho_{(12\dots m)}\;.
$$
\end{conjecture}
The results in~\cite{en2003}, determining $\rho_{\sigma}$ for any
$\sigma\in \S_3$ settle in the affirmative the CMP conjecture for patterns of length
$3$. 
Elisalde and Noy in~\cite{en2012} show that the conjecture is true for the
class of non-overlapping patterns. 
They also study the monotone pattern and provide the exact value of
$\rho_{(12\dots m)}$ implicitly as the smallest root of a formal power series. 

Regarding the least avoided pattern among all the patterns of length $m$, Nakamura in~\cite{n2012}
posed the following conjecture,
\begin{conjecture}[\cite{n2012}]\label{conj:least avoided pattern}
 For any $\sigma\in \S_m$,
$$
\rho_{\sigma}\geq \rho_{(12\dots m-2, m,m-1)}\;.
$$
\end{conjecture}

Both conjectures have been recently proved by Elizalde in~\cite{e2012}. The
proofs are based on computing the generating function for the number of $\sigma$-avoiding
permutations, $P_\sigma(z)=\sum \alpha_n(\sigma)
\frac{z^n}{n!}$, combined with the cluster method of Goulden and Jackson~\cite{gj1983}.

% , and using the
% fact that $\alpha_n(\sigma)=\Theta(\rho_0^n n!)$, where $\rho_0$ is the smallest root of
% $P_{\sigma}(z)$. Moreover, he gives an implicit exact expression for $\rho_{\sigma}$ for
% any pattern $\sigma\in \S_m$ as the smallest solution of $P_{\sigma}(z)$. 

Here we will use a complete different approach to the consecutive pattern avoiding problem through
the so called probabilistic method. While this approach is not as precise as the generating
function technique, it is simpler. This means that it provides more direct proofs of some existing
results, such as the CMP conjecture, and indeed, it allows us to go further in some directions, as
will be seen in Section~\ref{sec:whp}.

Our first result bounds from above $\rho_{\sigma}$ when the pattern $\sigma$ is
not monotone.
\begin{theorem}\label{thm:UB}
For any $\sigma \in \S_m\setminus \{(12\dots m),(m\dots 21)\}$,
$$
\rho_{\sigma}\leq 1-\frac{1}{m!}+O\left(\frac{1}{m^2 m!}\right) \;.
$$
\end{theorem}
The proof of this and all the following results, have a probabilistic flavor. We set out the problem
through the Poisson Paradigm (see e.g.~\cite{as2008}) which asserts that, in a probability space, events that are nearly
independent should behave similar as if they were so. To prove this theorem we make use of the
Suen's Inequality~\cite{s1990}, a powerful tool that provides an upper bound on the probability that
some events do not happen at the same time. 

Theorem~\ref{thm:UB} can be extended to the whole set of patterns, $\S_m$, by weakening the upper
bound: for any $\sigma\in\S_m$,
$$
\rho_{\sigma}\leq 1-\frac{1}{m!}+O\left(\frac{1}{m\cdot m!}\right) \;,
$$
however, this bound is not strong enough to prove the CMP conjecture.
From the results given in~\cite{en2012}, one can derive a lower bound on $\rho_{(12\dots m)}$ to
show that the CMP conjecture holds for any large enough $m$. 
% The main drawback of our approach to CMP
% conjecture is that the length of the pattern is required to be large enough. 
A more careful analysis on the constants hidden inside the asymptotic notation shows that it is
enough to consider $m\geq 5$.

The second part of the article is devoted to give a general lower bound on $\rho_{\sigma}$ when
$\sigma\in \S_m$.
\begin{theorem}\label{thm:LB}
For any $\sigma \in \S_m$,
$$
\rho_{\sigma}\geq 1-\frac{1}{m!}-O\left(\frac{m-1}{(m!)^2}\right)\;.
$$
\end{theorem}
To prove this lower bound we use a one--sided version of the Lov\'asz Local
Lemma~(see~\cite{ps2010}). This bound is asymptotically tight and an extremal example is provided
by the pattern $(12\dots m-2,m,m-1)$.
Unlike in the case of the upper bound and the CMP conjecture, the proof of
Theorem~\ref{thm:LB} can not be adapted to extract a proof of Conjecture~\ref{conj:least avoided
pattern}.

As Theorem~\ref{thm:UB} and Theorem~\ref{thm:LB} give bounds for the value of
$\rho_{\sigma}$ in terms of $m$, a natural question is to determine how do most of
the patterns behave. In this direction a much stronger upper bound, close to the general lower
bound, is shown to hold for most of the patterns.
\begin{theorem}\label{thm:whp} 
Let $\sigma\in \S_m$ chosen uniformly at random. Then, for any $2\leq k\leq m/2$, 
$$
\rho_{\sigma}\leq 1-\frac{1}{m!} +O\left(\frac{4^{m}}{\left(m-k\right)!m!}\right)\;,
$$
with probability at least $1-\frac{2}{(k+1)!}-m2^{-m/2}$.
\end{theorem}
This theorem shows that most of the patterns behave similar to the least avoided one.
% gives an upper bound on the probability that $\rho_{\sigma}$ is large for a
% $\sigma\in\S_m$ chosen uniformly at random. Moreover, it
The idea behind this result is that the number of permutations avoiding a pattern depends on the
maximum overlapping position of this pattern. It can be shown that almost all patterns do not have
a large overlap and thus, they are far from the upper bound attained by monotone patterns, the
ones with maximum overlap.

This paper is organized as follows. In Section~\ref{sec:UB}, Theorem~\ref{thm:UB} is proven. A
lower bound on $\rho_{(12\dots m)}$ is derived in
Section~\ref{sec:monotone} completing the proof of the CMP conjecture. Section~\ref{sec:LB} is
devoted to the proof of Theorem~\ref{thm:LB}. Finally, in
Section~\ref{sec:whp} we provide the proof of
Theorem~\ref{thm:whp}.

\section{An upper bound on $\rho_{\sigma}$}\label{sec:UB}

Consider the set of events $\A=\{A_1,\dots,A_N\}$ with associated indicator random variables
$X_1,\dots,X_N$ and let $X=\sum_{i=1}^N X_i$. In general, the events in $\A$ will be considered to
be bad and the aim is to bound, either from above or from below, the probability that none of these
bad events occurs. We will denote by $\mu$ the expected number of bad
events, this is, $\mu= \mathbb{E}(X) = \sum_{i=1}^N \Pr(A_i)$.

A \emph{dependency graph} of $\A$ is a graph $H$ with vertex set
$V(H)=\{1,\dots,N\}$ where
if two disjoint subsets $S,T\subseteq [N]$ share no edges then $\{A_i\}_{i\in S}$ and
$\{A_j\}_{j\in T}$ are independent. 

Two parameters are defined to control the dependencies
among all events. To measure the global effect of the dependencies, consider $$\Delta=\sum_{ij\in
E(H)}\Pr(A_i\cap A_j)\;,$$ and for the local
one, $$\delta=\max_{1\leq i\leq N}\sum_{j:\,ij\in E(H)}\Pr(A_j)\;,$$
where $E(H)$ denotes the edge set of $H$.

We will use the following version of Suen's inequality (see e.g. Theorem~2 in~\cite{j1998}),
\begin{theorem}[Suen's inequality] \label{thm:suen}
% Define 
% \begin{eqnarray*}
%  \mu &=& \sum_{i=1}^N \Pr(A_i),\\
% \Delta &=& \sum_{ij\in E(H)}\Pr(A_i\cap A_j) \text{ and}\\
% \delta&=&\max_{1\leq i\leq N}\sum_{j:\,ij\in E(H)}\Pr(A_j).\\
% \end{eqnarray*}
% For any $q\in [0,1]$,
% \begin{eqnarray}\label{eq:SuenUB2}
% \Pr(X=0)=\Pr(\cap \overline{A_i}) \leq \exp\left\{ -q\mu+q^2\Delta e^{2q\delta}\right\}.
% \end{eqnarray}
% In particular, by setting $q=\min\left(\frac{\mu}{4\Delta},\frac{1}{3\delta},1\right)$ we have
% \begin{eqnarray*}
% \Pr(X=0)=\Pr(\cap \overline{A_i}) \leq \exp\left\{ -\min\left(
% \frac{\mu^2}{8\Delta},\frac{\mu}{2},\frac{\mu}{6\delta}\right)
% \right\}.
% \end{eqnarray*}
With the above notation,
\begin{eqnarray}\label{eq:SuenUB2}
\Pr(X=0)=\Pr(\cap \overline{A_i}) \leq \exp\left\{ -\left(1-\frac{\Delta
e^{2\delta}}{\mu}\right)\mu\right\}.
\end{eqnarray}
\end{theorem}
% Equation~\eqref{eq:SuenUB2} follows from the proof of Suen's inequality (see
% e.g. Equation~$8$ in~\cite{j1998}).

Suen's inequality bounds from above the probability of having no bad events in terms of the expected
number of events, but also takes into account the pairwise dependence of the events. Thus, if the
dependencies among the events are weak or unlikely, we will be able to give a meaningful upper bound
on such probability.

Let $\pi\in\S_n$ chosen uniformly at random, and let $\sigma\in \S_m$ be a fixed pattern. For any
$0\leq i\leq n-m$ we define the event $A_i:=\{\st(\pi_{i+1},\dots,\pi_{i+m})=\sigma\}$. Then $\pi$
avoids $\sigma$ as a consecutive pattern if and only if $X=0$, this is, no copy
of the pattern $\sigma$ appears. By computing the probability of this event, 
$$
\alpha_n(\sigma)=\Pr(X=0)n!\;,
$$
where $\Pr(X=0)$ depends on $\sigma$. In particular we will be interested in 
\begin{eqnarray}\label{eq:relation}
\rho_{\sigma}= \lim_{n\to \infty} \Pr(X=0)^{1/n}\;. 
\end{eqnarray}

Bounding from above the number of edges in a dependency graph $H$ is crucial in order to give a
proper upper bound on the probability that any of the events happens at the same time.
The following lemma shows that there are many pairs of sets of events
that share no edges.
\begin{lemma}\label{lem:indep}
 Let  $S,T\subseteq \{0,1,\dots, n-m\}$ be two disjoint subsets of indexes such that for any
$i\in S$ and any $j\in T$, we
have $|i-j|\geq m$. Then,
the events $\{A_i\}_{i\in S}$ and $\{A_j\}_{j\in T}$ are independent.
\end{lemma}
\begin{proof}
For any two disjoint sets $U^+,U^-\subseteq \{0,1,\dots, n-m\}$, define the event 
$$
A_{U^+,U^-} :=\left\{\bigcap_{i\in U^+} A_i
\wedge \bigcap_{i\in U^-} \overline{A_i}\right\}\;.
$$
It suffices to show that for any two disjoint subsets $S^+,S^-\subseteq S$ and $T^+,
T^-\subseteq T$,
\begin{eqnarray}\label{eq:1}
\Pr\left( A_{S^+,S^-}\mid A_{T^+,T^-} \right)= \Pr\left( A_{S^+,S^-}\right)\;.
\end{eqnarray}

% Let $\supp(S)$ be the set of indexes in $\{0,1,\dots, n-m\}$ at distance at most $m-1$ from some
% element of $S$. 

We say that $j\in \{0,1,\dots, n-m\}$ belongs to the support of $S$, $\supp(S)$, if there
exists $i\in S$ such that $0\leq j-i \leq m-1$. Observe that $\supp(S)\cap \supp(T)=\emptyset$, by the assumptions on $S$ and $T$.
Clearly, the event $A_{S^+,S^-}$ is determined by the elements appearing in the positions indexed by
$\supp(S)$.

Denote by $\T \subseteq \S_n$ the subset of permutations of length $n$ that satisfies $A_{T^+,T^-}$.
Choose $\tau\in \T$ uniformly at random. It is enough to consider $\tau$ restricted on $\supp(S)$,
$\tau'$, and show that its standardization, $\st(\tau')$, is uniformly distributed in
$\S_{|\supp(S)|}$. 

The key observation is that $A_{T^+,T^-}$ might condition which elements lie in
$\supp(S)$ but does not impose anything on their order. The event $A_{T^+,T^-}$ makes no direct
restriction affecting the order of the elements in $\supp(S)$. Therefore, the elements appearing in
$\tau'$ may be conditioned by $A_{T^+,T^-}$, but $\st(\tau')$ is not affected by $A_{T^+,T^-}$.
Since $A_{S^+,S^-}$ is satisfied in $\tau'$ if and only if, it is satisfied in $\st(\tau')$ (with
the corresponding relabeling), equation~\eqref{eq:1} holds.
\end{proof}

The previous lemma suggests that a good dependency graph for the set of events $\A$ is the
circulant
graph $H$ with vertex set $V(H)=\{0,1,\dots,n-m\}$, where $ij\in E(H)$ if and only if $0<|i-j|<m$.
Throughout the paper, we will use the former circulant graph as a dependency graph of $\A$.

A simple upper bound follows directly from the previous observation. Consider $I= \{km:\, 0\leq
k<n/m\}$, then
$$
\Pr(X=0)=\Pr\left(\bigcap_{i=0}^{n-m} \overline{A_i}\right)  \leq \Pr\left(\bigcap_{i\in I}
\overline{A_i}\right)
=\prod_{i\in I} \left(1-\Pr\left(A_i\mid \bigcap_{j\in I,j<i} \overline{A_j}  \right)\right)\;.
$$
By using Lemma~\ref{lem:indep} with $S=\{i\}$ and $T= \{j:\;j\in I,j<i\}$,
$$
1-\Pr\left(A_i\mid \bigcap_{j\in I,j<i} \overline{A_j}  \right) =
1-\Pr\left(A_i\right)=
1-\frac{1}{m!}\;.
$$
Since $|I|\leq n/m$, this implies 
$$
\rho_{\sigma}\leq \left(1-\frac{1}{m!}\right)^{1/m} = 1-O\left(\frac{1}{m\cdot m!}\right)\;.
$$
However, a  better bound is given in Theorem~\ref{thm:UB} by taking into account the
interaction between pairs of dependent events.

A pattern $\sigma\in \S_m$ has an \emph{overlap at $k$} if
$\st(\sigma_{1},\dots,\sigma_{k})=\st(\sigma_{m-k+1},\dots,\sigma_{m})$, this is, the first and the
last $k$ positions have the same relative order. If a pattern does not have an overlap at $k$, then
\begin{eqnarray}\label{eq:incompatible}
\Pr(A_i\cap A_{i+m-k})=0\;.
\end{eqnarray}

For any $1 \leq k\leq m-1$, define the set $\M_k\subseteq \S_m$ as the set of patterns of length $m$
that have no overlap larger than $k$. The elements in $\M_1$
are called \emph{non-overlapping} patterns. They have been enumerated in~\cite{b2012} and also
extensively studied in~\cite{e2012}.

Observe that $\M_{m-1}$ is the whole set of patterns of length $m$.
One of the crucial facts to prove Theorem~\ref{thm:UB} is to show that $\M_{m-1}\setminus\M_{m-2}$,
the set of patterns that have an overlap at $m-1$, only consists of the monotone patterns.
\begin{lemma}\label{lem:monotone}
 For any $m\geq 3$,
$$
\M_{m-1}\setminus\M_{m-2}= \{(12\dots m),(m\dots 21)\}\;.
$$
\end{lemma}
\begin{proof}
 It is clear that both monotone patterns belong to $\M_{m-1}\setminus\M_{m-2}$. 
% Indeed, they overlap at any position. 
Let us show that any other
$\sigma \in \S_m \setminus \{(12\dots m),(m\dots 21)\}$ does not.
Suppose that $\sigma\in\M_{m-1}\setminus\M_{m-2}$. This implies that
\begin{eqnarray}
\label{condition for M_{m-1}}
\st(\sigma_1\dots\sigma_{m-1})=\st(\sigma_2\dots\sigma_{m})\;.
\end{eqnarray}
 Since $\sigma$ is not a monotone pattern, there exists an index $2\leq
i\leq m-1$ such that $\sigma_{i-1}>\sigma_{i}<\sigma_{i+1}$ or
$\sigma_{i-1}<\sigma_{i}>\sigma_{i+1}$.
Without loss of generality we assume the latter. Now observe that \eqref{condition for
M_{m-1}} implies that if $\sigma_{i-1}<\sigma_{i}$, then $\sigma_{i}<\sigma_{i+1}$, leading a
contradiction.
\end{proof}
Thus, we can consider that the maximum overlap of a pattern $\sigma \in\S_m \setminus \{(12\dots
m),(m\dots 21)\}$ is at most at $m-2$. This can not be improved since there are non monotone
patterns that have an overlap at $m-2$. For instance, consider $m=2t$ and
$\sigma=(1,t+1,2,t+2,\dots,t,2t)$.

% The following lemma observes that given a pattern $\sigma\in \S_m$ which has an overlap at $k$ and
% a pattern $\tau\in\S_{2m-k}$ such that the events $A_0$ and $A_{m-k}$ hold, then $\tau$ have
% some fixed.

The following lemma gives some insight of the structure of the permutations that contain two
given occurrences of a pattern $\sigma$. 
\begin{lemma}\label{lemma:fixed}
Let $\sigma\in \S_m$ be a pattern with an overlap at $k$ and suppose that $\tau\in\S_{2m-k}$ is such
that the events $A_0$ and $A_{m-k}$ hold. If $\sigma'=\st(\sigma_{m-k+1},\dots,\sigma_{m})$, then,
for any $0\leq i < k$, we have $\tau_{m-i}=\sigma_{k-i}+\sigma_{m-i}-\sigma'_{k-i}$.
\end{lemma}
\begin{proof}
 Fix some $i< k$. By the event $A_0$, we know that $\tau_{m-i}$ should be larger than
$\sigma_{m-i}-1$
elements and smaller than $m-\sigma_{m-i}$ elements from
$(\tau_1,\dots,\tau_{m-i-1},\tau_{m-i+1},\dots, \tau_{m})$. By the event $A_{m-k}$, it is
also true that $\tau_{m-i}$ is larger than $\sigma_{k-i}-1$ and smaller than $m-\sigma_{k-i}$
elements from $(\tau_{m-k+1},\dots,\tau_{m-i-1},\tau_{m-i+1},\dots, \tau_{2m-k})$. 

Consider now the permutation $\sigma'=\st(\sigma_{m-k+1},\dots,\sigma_{m})\in \S_{k}$. Then
there are $\sigma'_{k-i}-1$ elements that are counted twice when we look at the
elements smaller than $\sigma_{m-i}$ or $\sigma_{k-i}$, and $k-\sigma'_{k-i}$ also double counted
when we look to the larger ones. Therefore
\begin{eqnarray*}
\tau_{m-i}&>& \sigma_{k-i}+\sigma_{m-i}-2- (\sigma'_{k-i}-1)\;,
\end{eqnarray*}
and
\begin{eqnarray*}
\tau_{m-i}&\leq& 2m-k-\left(m-\sigma_{k-i}+m-\sigma_{m-i}- (k-\sigma'_{k-i})\right)\;.
\end{eqnarray*}
Observing that the first inequality is strict,
$$\tau_{m-i}=\sigma_{i+1}+\sigma_{m-i}-\sigma'_{k-i}.$$
\end{proof}

Using this last lemma, we can provide an upper bound on the probability that two given occurrences
of a pattern appear.
\begin{lemma}\label{lem:bound1}
	For any $\sigma\in \S_m$ and any $1\leq k\leq m-1$,
	$$
	\Pr(A_i\wedge A_{i+m-k})\leq \frac{4^{m-k}}{\sqrt{\pi(m-k)}\cdot(2m-k)!}\;.
	$$
\end{lemma}
\begin{proof}
If $\sigma$ does not have an overlap at $k$, $\Pr(A_i\wedge A_{i+m-k})=0$ and we are done. Thus,
assume that $\sigma$ has an overlap at $k$.

Set $\tau=\st(\pi_{i+1},\dots, \pi_{i+2m-k})$. 
Recall that $\pi\in \S_n$ has been chosen uniformly at random, which implies that $\tau$ is uniformly distributed in
$\S_{2m-k}$. Moreover, $\pi$ satisfies $A_i$ and $A_{i+m-k}$ if and only if $\tau$ satisfies $A_0$ and $A_{m-k}$.

There are $(2m-k)!$ possible candidates for $\tau$. 
We will count how many of them are such that the events $A_0$ and
$A_{m-k}$ hold. By Lemma~\ref{lemma:fixed}, we know that the elements
$(\tau_{m-k+1},\dots,\tau_m)$ are
uniquely determined by $\sigma$ and $k$. Thus, one must select a subset of $m-k$ elements among the
$2m-2k$ available ones, to construct $(\tau_{1},\dots,\tau_{m-k})$. Since $\tau$ satisfies $A_0$,
once these elements have been chosen, there is just one order such that
$\st(\tau_{1},\dots,\tau_{m})=\sigma$, and only one way to set the last $m-k$ elements of $\tau$, in
order to satisfy $A_{m-k}$.

Hence, for $\pi\in \S_n$, 
	$$
	\Pr(A_i\wedge A_{i+m-k})\leq \frac{\binom{2(m-k)}{m-k}}{(2m-k)!} \leq 
\frac{4^{m-k}}{\sqrt{\pi(m-k)}\cdot(2m-k)!}\;.
	$$
where we have used that $\binom{2a}{a}\leq\frac{4^a}{\sqrt{\pi a}}$. One can prove this last
inequality by using Stirling's approximation.
\end{proof}

Now we are able the proof the main theorem.
\begin{proof}[Proof of Theorem~\ref{thm:UB}]
First of all we compute $\mu$, $\Delta$ and $\delta$, needed to apply Suen's inequality. 
The expected number of occurrences of the pattern $\sigma$ does not depend on $\sigma$ and can be
computed as
$$
\mu = \sum_{i=0}^{n-m} \Pr(A_i) = \frac{n-m+1}{m!}\leq \frac{n}{m!}\;.
$$

Recall that by the choice of the dependency graph $H$ (inspired by Lemma~\ref{lem:indep}) two events
$A_i$ and $A_j$ share no edge if $|i-j|\geq m$. Assume that $i<j$ and
$j-i=m-k$, then by Lemma~\ref{lem:bound1},
\begin{eqnarray*}
\Pr(A_i\wedge A_j) \leq& \frac{4^{m-k}}{\sqrt{2\pi}(2m-k)!} & \text{ if }k\leq m-2\;.\\
\end{eqnarray*}
and, since $\sigma$ is not monotone, by \eqref{eq:incompatible} and Lemma~\ref{lem:monotone},
\begin{eqnarray*}
\Pr(A_i\wedge A_j) =&0 & \text{ if }k=m-1\;.
\end{eqnarray*}

Hence,
\begin{eqnarray}\label{eq:explicit_Delta}
\sum_{j=i+1}^{i+m-1} \Pr(A_i\wedge A_j) &\leq& \sum_{k=1}^{m-2} \frac{4^{m-k}}{\sqrt{2\pi}(2m-k)!}\\
%\frac{16}{\sqrt{2\pi}(m+2)!}+\frac{64}{\sqrt{2\pi}(m+3)!}+\dots+\frac{4^{m-1}}{\sqrt{2\pi}(2m-1)!}
&=& \left(1+\frac{4}{m+3}+
O(m^{-2})\right)\frac{16}{\sqrt{2\pi}(m+2)!} \leq \frac{17}{\sqrt{2\pi}(m+2)!}\;,\nonumber
\end{eqnarray}
for any $m$ large enough.

Then, $\Delta$ can be expressed as
$$
\Delta = \sum_{0=i}^{n-m} \sum_{j=i+1}^{i+m-1} \Pr(A_i\wedge A_j) \leq
\frac{17n}{\sqrt{2\pi}(m+2)!}\;.
$$

Since the degree of a vertex in the dependency graph $H$ is at most $2(m-1)$,
\begin{eqnarray*}
\delta&=& \max_{0\leq i\leq n-m} \sum_{j:\,ij\in E(H)} \Pr(A_j) = 2(m-1)\Pr(A_j)=
\frac{2(m-1)}{m!}\leq
\frac{2}{(m-1)!}\;.
\end{eqnarray*}
% We will set
% $$q=\min\left(\frac{\mu}{c\Delta},\frac{1}{b\delta},1\right)$$
% for some $b,c>0$ that can depend on $m$. 
% 
% The best bound provided by~\eqref{eq:SuenUB2} is achieved when the three terms in the minimum are
% equal, i.e. we have $\frac{\mu}{c\Delta}= 1$ and $\frac{1}{b\delta}=1$.
% 
% Observe that since $q=1$, Equation~\eqref{eq:SuenUB2} becomes
% \begin{eqnarray}\label{eq:new_Suen}
% \Pr(X=0)\leq \exp\left(-\left(1- \frac{\Delta e^{2\delta}}{\mu}\right)\mu \right) \;.
% \end{eqnarray}

Using that $e^{2\delta}\leq e^{4/(m-1)!}\leq 2$ if $m\geq 4$, Suen's
inequality~(see~\eqref{eq:relation}) implies that for a large enough $m$
\begin{eqnarray*}\label{eq:finalUB}
\rho_{\sigma} &\leq& \exp\left(-\frac{ 1-\frac{34}{\sqrt{2\pi}(m+2)(m+1)}} {m!}\right)\\
&\leq& 1-\frac{\frac{1}{m!}-\frac{34}{\sqrt{2\pi}(m+2)(m+1)\, m!}}{1+\frac{1}{m!}}\\
&\leq&
1-\left(1-O\left(\frac{1}{m!}\right)\right)\left(\frac{1}{m!}-\frac{34}{\sqrt{2\pi}(m+2)(m+1)\,
m!}\right)\\
&\leq& 1-\frac{1}{m!}+\frac{14}{m^2\, m!}\;.
\end{eqnarray*}
for any large enough $m$. We have used that $e^{-a}\leq 1-\frac{a}{1+a}$, for any $a\geq 0$. 
\end{proof}
% Observe that the constant $14$ on Theorem~\ref{thm:UB} can be as close to $\binom{4}{2}=6$ as we
% need by considering $m$ to be large enough and being careful on not to lose too much in the
% inequalities. We make no effort to optimize the constants.

% Now we set $\rho_{\sigma}=1-\eps_{\sigma}$. Theorem~\ref{thm:UB} shows that if $\sigma \in
% \S_m\setminus\{(12\dots m),(m\dots 21)\}$ 
% $$
% \eps_\sigma\geq \frac{1}{m!}-\frac{35}{(m+2)(m+1)m!},
% $$
% which compared to~\eqref{eq:eps_12m} gives the following results
% 

\section{A probabilistic proof of CMP conjecture}\label{sec:monotone}
% 

% Let 
% $$
% \rho_{(12\dots m)} = \lim_{n\to \infty} \left(\frac{\alpha_n(12\dots m)}{n!}\right)^{1/n}
% $$
% then a conjecture of Elisalde and Noy states that $\rho_{\sigma}\leq \rho_{(12\dots m)}$ for any
% $\sigma\in \S_m$.

% This conjecture is known to be true for some families of permutations such as the non-overlapping.

In this section we aim to provide an alternative proof of the CMP conjecture. We do it by obtaining
a lower bound on $\rho_{(12\dots m)}$ and showing that this bound is larger than the upper bound
obtained in Theorem~\ref{thm:UB}. A recent result of Elisalde and Noy gives an
implicit expression for $\rho_{(12\dots m)}$.
\begin{theorem}[Elisalde and Noy~\cite{en2012}]
 Let $z_0=\rho_{(12\dots m)}^{-1}$, then $z_0$ is the smallest solution of $$
g(z) = \sum_{i\geq 0} \frac{z^{mi}}{(mi)!}-\sum_{i\geq 0} \frac{z^{mi+1}}{(mi+1)!}\;.
$$
\end{theorem}

From this last theorem we can extract an explicit lower bound on $\rho_{(12\dots m)}$.
\begin{lemma}\label{lem:cor_en}
 For any $m$ large enough,
$$
\rho_{(12\dots m)}\geq 1-\frac{1}{m!} +\frac{1}{m\cdot m!}+
O\left(\frac{1}{m^2\cdot m!}\right)\;.
$$ 
\end{lemma}
\begin{proof}
Observe that
$$
f(z)=1-z+\frac{z^m}{m!}-\frac{z^{m+1}}{(m+1)!}+ \frac{z^{2m}}{(2m)!}\geq g(z)\;,
$$
since $g(z)$ is an alternating sum whose terms are strictly decreasing. Since $g(0)=1$ and $z_0$ is
the smallest root of $g(z)$ we can conclude that $z_1$, the smallest root of $f(z)$, is at
least $z_0$. Thus $ \rho_{(12\dots m)}\geq 1/z_1$ and it suffices to compute an upper bound on
$z_1$.

Write $z= (1-\eps)^{-1}$, then $z^{-2m}f(z)=0$ becomes 
$$
-(1-\eps)^{2m-1}\eps+\frac{(1-\eps)^{m-1}}{(m+1)!}(m-(m+1)\eps)+\frac{1}{(2m)!}=0\,.
$$
Using $ 1-nx\leq (1-x)^n\leq 1-nx+n^2x^2$,

\begin{eqnarray}\label{eq:explicit_monotone}
0&\leq& -(1-(2m-1)\eps)\eps +\frac{1-(m-1)\eps+ (m-1)^2 \eps^2}{(m+1)!}(m-(m+1)\eps)
+\frac{1}{(2m)!}\nonumber\\
&\leq &\left(2m-1+\frac{(m-1)(m^2+1)}{(m+1)!}\right)\eps^2-\left(1+\frac{m^2+1}{(m+1)!}
\right)\eps+\left(\frac{m}{(m+1)!}+\frac{1}{(2m)!}\right) \;.
\end{eqnarray}

Let $\eps'$ be the solution of the last equation with equality. Then, $\rho_{(12\dots m)}\geq
(1-\eps')$. If $m$ is large enough we can get an asymptotic expression for $\eps'$.
Suppose that $b^2\gg 4ac$, then the smallest solution of $ax^2+bx+c=0$ can be
approximated by 

\begin{eqnarray}\label{eq:sec_deg_eq}
x=-\frac{c}{b}-\frac{ac^2}{b^3} + O\left(\frac{a^2c^3}{b^5}\right)\;.
\end{eqnarray}

This leads to
\begin{eqnarray*}
\eps'&=& \frac{\frac{m}{(m+1)!}+\frac{1}{(2m)!}}{1+\frac{m^2+1}{(m+1)!}}+
O\left(\frac{m^3}{(m+1)!^2}\right)\\
&= & \frac{m}{(m+1)!}+
O\left(\frac{m^3}{(m+1)!^2}\right)\\
&= & \frac{1}{m!\left(1+\frac{1}{m}\right)}+
O\left(\frac{m^3}{(m+1)!^2}\right)\\
&= & \frac{1}{m!} -\frac{1}{m\cdot m!}+
O\left(\frac{1}{m^2\cdot m!}\right)\; ,\\
\end{eqnarray*}
where we have used $(1+x)^{-1}=1-x+O(x^2)$ in the second and the last inequalities. This proves the
lemma.
\end{proof}

% Elizalde and Noy~\cite{en2012} have recently computed the exact value of $\rho_{(12\dots m)}$. In
% particular, one can conclude that
% $$
% \rho_{(12\dots m)}=1-\frac{1}{m!}+\Omega\left(\frac{1}{m\, m!}\right).
% $$

The CMP conjecture comes as an straightforward corollary of
Theorem~\ref{thm:UB} together with Lemma~\ref{lem:cor_en}.
\begin{corollary}\label{cor:conjecture}
For any large enough $m$, the CMP conjecture is true. Moreover, for any $\sigma \in 
\S_m\setminus \{(12\dots m),(m\dots 21)\}$,
$$
1-\rho_{\sigma}\geq \left(1+\frac{1}{m}\right)(1-\rho_{(12\dots m)})\;.
$$
\end{corollary}

Thus, this corollary does not only show that the CMP conjecture is true, but also provides a lower
estimation of the minimum gap between $\rho_{(12\dots m)}$ and $\rho_{\sigma}$, for any $\sigma \in 
\S_m\setminus \{(12\dots m),(m\dots 21)\}$.

Note that the last corollary holds for any $m$ large enough. An upper bound without any assumption on
$m$ that can be derived from~\eqref{eq:SuenUB2} and~\eqref{eq:explicit_Delta} in
Theorem~\ref{thm:UB}. Comparing this bound with the exact lower bound that follows
from~\eqref{eq:explicit_monotone} in Lemma~\ref{lem:monotone}, one can check that the CMP conjecture
holds for any $m\geq 5$.

As we will see in the next section, it is also possible to provide a lower bound on $\rho_{\sigma}$
in
terms of $m$ without using generating functions. Unlike the upper bound case, the lower bound just
takes into account the number of dependencies among the events, but not the nature of these
dependencies. It might be interesting to give a direct proof of the lower bound for the monotone
pattern (Lemma~\ref{lem:monotone}), which does not rely upon any other result. For such a purpose
it would be useful to understand the probabilities $\Pr\left(A_i\mid \bigcap_{j<i}
\overline{A_j}\right)$ when $\sigma=(12\dots m)$.

\section{A lower bound on $\rho_{\sigma}$.}\label{sec:LB}

The setting used to give an upper bound to the number of permutations avoiding a given pattern can
be also used to provide a lower bound on $\rho_{\sigma}$. Now we need a way to bound from
below the probability that $X=0$ and for such a purpose we will use the Lov\'asz Local
Lemma. 

Usually, the Local Lemma is used to show the existence of a certain configuration that does
not satisfy any of the bad events in $\A$. However, in our problem it is trivial to see that for
any pattern $\sigma\in \S_m$ there exists at least one permutation of length $n$ that avoids
$\sigma$.
Nevertheless, it also provides an explicit lower bound on the probability that such configuration
exists, giving a lower estimation on the number of such configurations. We will use it to derive a
lower bound on the number of permutations of length $n$ that avoid $\sigma$.

The following version of the Local Lemma was proposed by Peres and Schlag in~\cite{ps2010} and it
is convenient for our approach. 
 \begin{lemma}[One--sided Local Lemma]
 \label{lem:OSLLL}
  Let $x_1,x_2,\dots,x_N$ be a sequence of numbers in $(0, 1)$. Assume that, for every $i\in N$,
there is an integer $0 < m(i) \leq i$ such that
\begin{eqnarray}\label{eq:cond}
\Pr\left(A_i\mid \bigcap_{j<m(i)} \overline{A_j}\right)\leq x_i\prod_{k=m(i)}^{i-1}(1-x_k)\;.
\end{eqnarray}
Then,
\begin{eqnarray}\label{eq:OSLLL}
\Pr(X=0) = \Pr \left(\bigcap_{i=1}^N \overline{A_i}\right) \geq \prod_{i=1}^N (1-x_i)\;.
\end{eqnarray}
% In particular, if $x_i=x$ for all $i$, then $x$ is the smallest solution of the equation
% $p=x(1-x)^d$, where $d$ is the maximum degree of the dependency graph.
\end{lemma}

To use the Local Lemma a dependency graph on the set of events must be set. In
the case of the one-sided version, the graph is defined implicitly in~\eqref{eq:cond} as the
directed circulant graph with out-degree $i-m(i)$. Thus, the same dependency graph used for Suen's
inequality is also valid to apply the Local Lemma. 

Next, we give the proof of the lower bound on
$\rho_{\sigma}$.
\begin{proof}[Proof of Theorem~\ref{thm:LB}]
Let $\A=\{A_0,\dots,A_{n-m}\}$ and $X$ be defined as in Section~\ref{sec:UB}.
Set $m(i)=i-m+1$. Using Lemma~\ref{lem:indep} with $S=\{i\}$ and $T=\{0,1\dots,i-m\}$
\begin{eqnarray}\label{eq:0}
\Pr\left(A_i\mid \bigcap_{j\leq i-m} \overline{A_j}\right) = \Pr\left(A_i\right)\;.
\end{eqnarray}

Since all the events are symmetric we set $x_i=x$, for any
$0\leq i\leq n-m$. Then, condition~\eqref{eq:cond} becomes
\begin{eqnarray}\label{eq:simpl_cond} 
\Pr\left(A_i\right) \leq x (1-x)^{m-1}\;.
\end{eqnarray}
Recall that $\Pr\left(A_i\right)=\frac{1}{m!}$.
Thus, the previous equation implies that $x>\frac{1}{m!}$. Besides, we are interested on
keeping $x$ as small as possible, because of~\eqref{eq:OSLLL}. Let us write 
$x=\frac{e^{f(m)}}{m!}$ for some positive function $f(m)$.  Hence, using $(1-x)\leq
e^{-x}$, condition~\eqref{eq:simpl_cond} implies
\begin{eqnarray*}
	\frac{e^{f(m)}(m-1)}{m!} &\leq& f(m)\;,
\end{eqnarray*}
which also implies $f(m)\geq \frac{m-1}{m!}$, since $f(m)\geq 0$.
By setting $x=\frac{e^{\frac{m-1}{m!}}}{m!}$, condition~\eqref{eq:cond} is satisfied and the
Local Lemma can be applied. In particular, we obtain the
following lower bound on the probability that $X=0$,
$$
\Pr(X=0)=\Pr\left(\bigcap_{i=0}^{n-m} \overline{A_i}\right)\geq
\left(1-\frac{e^{(m-1)/m!}}{m!}\right)^{n-m+1}\;.
$$
and using~\eqref{eq:relation},
$$
\rho_{\sigma}\geq 1-\frac{e^{(m-1)/m!}}{m!}= 1-\frac{1}{m!}-
O\left(\frac{m-1}{(m!^2)}\right)\;.
$$
\end{proof}
\vspace{0.3cm}

The lower bound given by Theorem~\ref{thm:LB} is tight. This can be shown using a result of
Elizalde in~\cite{e2012}, where the author proved that the least avoided pattern is $(12\dots
m-2,m,m-1)$. The author also gives an implicit lower bound to $ z_0 = \rho_{(12\dots m-2, m, m-1
)}^{-1}$ as the smallest root of
$$
f(z)=1-z+\frac{z^{m}}{m!}-m\frac{z^{2m+1}}{(2m-1)!}\;.
$$
An explicit upper bound can be derived from the previous equation, as in Lemma~\ref{lem:monotone}.
$$
\rho_{(12\dots m-2,m,m-1)} \leq 1-\frac{1}{m!} - O\left( \frac{m-1}{(m!)^2}\right)\;.
$$ 

In order to prove Conjecture~\ref{conj:least avoided pattern}, one could try to use the same
strategy we have used for the CMP conjecture. First, determine
the subset of patterns $\sigma$
such that $\alpha_n(\sigma)=\alpha_n(12\dots m-2,m,m-1)$ and finally, improve the lower bound for
the patterns which are not in the previous subset.
However, this approach is hopeless to tackle Conjecture~\ref{conj:least avoided pattern}. Notice
that no assumption on the properties of the
pattern has been used in the proof of the lower bound, like in the proof of the upper bound in
Theorem~\ref{thm:UB}. 
Unfortunately, the Local Lemma can not distinguish the different nature of the dependencies among
events. Thus, no better lower bound can be achieved by restricting
to a smaller subset of patterns. This is also the main problem to prove
Lemma~\ref{lem:monotone} using our approach.

In the next section we will improve the upper bound of Theorem~\ref{thm:UB} for large subsets of
patterns.

% 
% In a general setting
% (assuming that we don't know anything from $\sigma$) these dependencies
% are not clear, but, maybe, this bound can be improved if we assume that $\sigma$ is a given
% permutation (for instance $(12\dots m)$) or belongs to a special family of permutations (such that
% the non-overlapping ones). In any of the previous cases, since the local lemma can not distinguish
% the different type of dependence, something else should be used.

\section{The typical behavior of patterns.}\label{sec:whp}
The results of the previous sections provide tight upper and lower bounds on
$\rho_{\sigma}$ for any $\sigma\in \S_m$.
In this section we want to show that, for a typical pattern, $\rho_{\sigma}$ lies much closer to the
lower bound than to the upper bound. This is, the number of $\sigma$-avoiding permutations of length
$n$, when $\sigma\in\S_m$ chosen uniformly at random, is closer to the number of permutations
that avoid
$(12\dots m-2,m,m-1)$ than to the number of permutations that avoid $(12\dots m)$.

Define $\N_k\subseteq \S_m$ as the set of patterns of length $m$ that overlap at
position $k$. The following lemma bounds from above the size of these sets.
\begin{lemma}\label{lem:bound size N_k}
Let $\sigma\in \S_m$ chosen uniformly at random, then
\begin{enumerate}
 \item $\Pr(\sigma\in \N_k)= \frac{1}{k!}$ if $2\leq 2k\leq m$.
 \item $\Pr(\sigma\in \N_k)\leq 2^{-m/2}$ if
$m< 2k \leq 2(m-1)$.
\end{enumerate}
\end{lemma}
\emph{Proof of $1.$} Choose $\sigma\in\S_m$ uniformly at random.
Recall that the condition for $\sigma\in\N_k$ is that
$\tau^1=\st(\sigma_1,\dots,\sigma_k)$ and 
$\tau^2=\st(\sigma_{m-k+1},\dots,\sigma_m)$ are equal. If $2k\leq m$, then $\tau^1$ and $\tau^2$ are
independent by Lemma~\ref{lem:indep} and uniformly distributed in $\S_k$. For any $\tau,\tau'\in \S_k$
$$
\Pr(\tau^1=\tau \mid  \tau^2=\tau'  ) = \Pr(\tau^1=\tau)\;.
$$
Thus, we can compute the exact probability of being in $\N_k$
$$
\Pr(\sigma\in \N_k) = \Pr(\tau^1=\tau^2) = \sum_{\tau \in \S_k} \Pr(\tau^1=\tau \wedge \tau^2=\tau)
= k! \Pr(\tau^1=\tau)^2 = \frac{1}{k!}\;.
$$\qed

\vspace{0.3cm}

\emph{Proof of $2.$}  Choose $\sigma\in\S_m$ uniformly at random. Partition the pattern
$\sigma$ in parts of length $m-k$ by defining
$\tau^i=\st(\sigma_{(m-k)(i-1)+1},\dots,\sigma_{(m-k)i})$ for any $1\leq
i\leq\lfloor\frac{m}{m-k}\rfloor$. Observe that, in order to have an overlap at $k$ we must have
$\tau_1=\tau_i$ for any $i>1$. This condition is clearly necessary but not sufficient for a pattern
to overlap at $k$.

Since $2k> m$, we have at least $\lfloor\frac{m}{m-k}\rfloor\geq 2$ parts. 
By the choice of $\sigma$, the permutations $\tau^i$ are uniformly distributed
in $S_{m-k}$, and by Lemma~\ref{lem:indep}, they are mutually independent. This implies,
$$
\Pr(\sigma\in \N_k)\leq \prod_{i>1} \Pr(\tau_i=\tau_1)=
\left(\frac{1}{(m-k)!}\right)^{\lfloor\frac{m}{m-k}\rfloor -1} \leq 2^{-m/2+1}\;,
$$
for any $k\leq m-2$. If $k=m-1$, $\N_{m-1}$ is the set of patterns with an overlap at $m-1$ and the
upper bound is directly implied by Lemma~\ref{lem:monotone}.
\qed

Unlike in the case when $k\leq 2m$, where we can determine exactly the size of $\N_k$, a non tight
upper bound is given when $k>2m$. Observe that the sets $\N_k$ cover all $\S_m$ but they are not a
partition of it. For instance, monotone patterns belong to all such sets, since they overlap at any
possible position. However, we conjecture that 
$$
|\N_k|\leq \frac{1}{k!}\;,
$$
for every $1\leq k\leq m-1$. 
% The intuition is that when $2k>M$ for most of the patterns the first
% $2k-m$ and the last $2k-m$ elements will not have the same relative order. Therefore it will be
% more difficult to make them match.
% \comment{redactar aixo altra vegada}

We use the previous lemma to give a lower bound on the size of $\M_k$, the set of patterns that have
no overlap at any positions larger than $k$.
\begin{lemma}\label{lem:bound on M_k}
Let $\sigma\in \S_m$ chosen uniformly at random.
Then, for any $1\leq k\leq m/2$,
$$
\Pr\left(\sigma\in\M_{k}\right)\geq 1-\frac{2}{(k+1)!}-m2^{-m/2}.
$$
\end{lemma}
\begin{proof}

	Observe that we can bound from below
	the size of $\M_{k}$ using the sets $\N_k$,
	\begin{eqnarray}\label{eq:relation k}
	|\M_{k}|&=& \left|\S_m\setminus \bigcup_{\ell=k+1}^{m-1} \N_{\ell}\right|\geq  m! -
	\sum_{\ell=k+1}^{m-1} |\N_{\ell}|\;.
	\end{eqnarray}
	
By Lemma~\ref{lem:bound size N_k}, for any $k$ such that $2k\leq m$,
$$
\sum_{\ell=k+1}^{m-1} \Pr(\sigma\in\N_{\ell}) \leq \frac{1}{(k+1)!}+\frac{1}{(k+2)!}+\dots +
\frac{1}{\lfloor
m/2\rfloor!}+\frac{m}{2}2^{-m/2+1}\;.
$$

Using the relation in~\eqref{eq:relation k} gives 
\begin{eqnarray*}
\Pr\left(\sigma\in \M_{k}\right)\geq 1- \sum_{\ell=k+1}^{m-1} \Pr(\sigma\in
\N_{\ell})&\geq&
1-\sum_{\ell=k+1}^{m/2}\frac{1}{\ell!} -m2^{-m/2} \geq
1-\frac{2}{(k+1)!}-m2^{-m/2}\;.
\end{eqnarray*}
\end{proof}

Recall that $\M_1$ corresponds to the set of non-overlapping
patterns. The proof of Lemma~\ref{lem:bound on M_k} implies that $|\M_1|\geq (3-e)m!$. This bound can
be refined. Indeed, B\'ona~\cite{b2012} showed that
$$
0.364098149 \leq \frac{|\M_1|}{m!}\leq 0.3640992743\;.
$$ 
The previous bound on $|\M_k|$ is clearly non sharp.
A better estimation of the size of $\N_k$ when $2k>m$, would help to understand the distribution of
$\rho_{\sigma}$ when $\sigma\in\S_m$ is chosen uniformly at random.

Next lemma shows that a better bound on $\Delta$ can be given if the pattern does
not have a large overlap.
\begin{lemma}\label{lem:bound2}
	For any $\sigma\in \M_k$, 
	$$
	\Delta \leq \frac{4^{m-k}}{(2m-k)!}n \;.
	$$
\end{lemma}
\begin{proof}
Since $\sigma\in\M_k$ we have $\Pr(A_i\wedge A_{i+m-j})=0$ for any $j$ such that $k<j\leq m-1$. Using
Lemma~\ref{lem:bound1}
\begin{eqnarray*}
\Delta &\leq& \sum_{i=0}^{n-m} \sum_{j=1}^{k}\Pr(A_i\wedge A_{i+m-j}) \\
& \leq  & n \sum_{j=1}^{k}\frac{4^{m-j}}{\sqrt{\pi(m-j)}(2m-j)!}\\
& \leq & \frac{4^{m-k}}{(2m-k)!}n\;.
\end{eqnarray*}
\end{proof}

\begin{proof}[Proof of Theorem~\ref{thm:whp}]
% By Lemma~\ref{lem:bound on M_k} the
% probability that a random chosen pattern $\sigma\in\S_m$ belongs to $\M_{k}$ is
%
Assume that $\sigma\in\M_{k}$.
It follows from Lemma~\ref{lem:bound2} that
$$
\frac{\Delta}{\mu} \leq \frac{4^{m-k} m!}{\left(2m-k\right)!} =
\frac{4^{m}}{\binom{2m-k}{m}\left(m-k\right)!} \leq \frac{4^{m}}{
\left(m-k\right)!}\;.
$$
% where we used that $\binom{a}{b}\geq \left(\frac{a}{b}\right)^b$.
Since $e^{2\delta}\leq e^{4/(m-1)!}\leq 2$ for any $m\geq 4$, using~\eqref{eq:SuenUB2} we can
derive the following upper bound,
$$
\Pr(X=0)\leq \exp\left(-
\frac{1-O\left(\frac{4^{m}}{\left(m-k\right)!}\right)}{m!}n\right)\;.
$$
From~\eqref{eq:relation},
$$
\rho_{\sigma} = \lim_{n\to\infty} \Pr(X=0)^{1/n}\leq 1-
\frac{1}{m!}+O\left(\frac{4^m}{\left(m-k\right)!m!}\right)\;,
$$
where we have used $e^{-a}\leq 1-\frac{a}{1+a}$.

This upper bound holds when $\sigma\in\M_k$, and this holds with probability at least
$1-\frac{2}{(k+1)!}-m2^{-m/2}$ when $\sigma$ is chosen uniformly at random, by Lemma~\ref{lem:bound
on M_k}.
\end{proof}

% \begin{proof}[Proof of Theorem~\ref{thm:whp}]
% Suppose that $\sigma\in \cup_{\ell\leq m/2}\M_{\ell}$. Then, Lemma~\ref{lem:bound2} implies
% that
% $\Delta \leq \frac{4^m}{((3m/2)!)}n.$
% Then
% $$
% \frac{\Delta}{\mu} \leq \frac{4^m m!}{\left(\frac{3m}{2}\right)!} =
% \frac{4^{m}}{\binom{3m/2}{m}\left(\frac{m}{2}\right)!} \leq \frac{4^{m}}{
% \left(\frac{m}{2}\right)!}\;.
% $$
% % where we used that $\binom{a}{b}\geq \left(\frac{a}{b}\right)^b$.
% Since $e^{2\delta}\leq e^{4/(m-1)!}\leq 2$ for any $m\geq 4$, using~\ref{eq:SuenUB2} one gets,
% $$
% \Pr(X=0)\leq \exp\left(-
% \frac{1-O\left(\frac{4^{m}}{\left(\frac{m}{2}\right)!}\right)}{m!}\right)
% $$
% Using $e^{-a}\leq 1-\frac{a}{1+a}$, we have
% $$
% \rho_{\sigma} = \lim_{n\to\infty} \Pr(X=0)^{1/n}\leq 1-
% \frac{1}{m!}+O\left(\frac{4^m}{\left(\frac{m}{2}\right)!m!}\right)
% $$
% \end{proof}

\begin{acknowledgement}
The author is grateful to Marc Noy and Oriol Serra for 
helpful discussions.
\end{acknowledgement}
\bibliography{patterns}
\bibliographystyle{amsplain}

\end{document}